\documentclass[reqno,11pt]{amsart}
\usepackage{multirow}
\usepackage{tikz}
\usepackage{tikz-cd}
\usepackage{subcaption}
\usepackage{pgfplots,amssymb,enumerate,graphics,hyperref}
\pgfplotsset{compat=1.16}
\usepackage[margin= 1in]{geometry}
\newtheorem{theorem}{Theorem}[section]
\newtheorem{lemma}[theorem]{Lemma}
\newtheorem{proposition}[theorem]{Proposition}
\newtheorem{corollary}[theorem]{Corollary}
\theoremstyle{definition}

\newtheorem{example}[theorem]{Example}
\newtheorem{remark}[theorem]{Remark}
\begin{document}
\title{R\'edei permutations with the same cycle structure}
\author{Juliane Capaverde, Ariane M. Masuda,  and Virg\'inia M. Rodrigues}
\date{\today}
\address{Departamento de Matem\'{a}tica Pura e Aplicada, Universidade Federal do Rio Grande do Sul, Avenida Bento Gonçalves, 9500, Porto Alegre, 91509-900 Brazil}

\email{juliane.capaverde@ufrgs.br, vrodrig@mat.ufrgs.br}

\address{Department of Mathematics, New York City College of Technology, CUNY, 300 Jay Street, Brooklyn, NY 11201 USA}

\email{amasuda@citytech.cuny.edu}

\thanks{The second author received support for this project
provided by a PSC-CUNY grant, jointly funded by The Professional Staff Congress and The City University of New York.}

\begin{abstract}
Let $\mathbb{F}_{q}$ be the finite field of order $q$, and
$\mathbb{P}^{1}(\mathbb{F}_{q}) = \mathbb{F}_{q}\cup \{\infty \}$. Write
$(x+\sqrt{y})^{m}$ as $N(x,y)+D(x,y)\sqrt{y}$. For $m\in \mathbb{N}$ and
$a \in \mathbb{F}_{q}$, the R\'{e}dei function
$R_{m,a}\colon \mathbb{P}^{1}(\mathbb{F}_{q}) \to \mathbb{P}^{1}(
\mathbb{F}_{q})$ is defined by $N(x,a)/D(x,a)$ if $D(x,a)\neq 0$ and
$x\neq \infty $, and $\infty $, otherwise. In this paper we give a complete
characterization of all pairs $(m,n)\in \mathbb{N}^{2}$ such that the R\'{e}dei
permutations $R_{m,a}$ and $R_{n,b}$ have the same cycle structure when
$a$ and $b$ have the same quadratic character and $q$ is odd. We explore
some relationships between such pairs $(m,n)$, and provide explicit families
of R\'{e}dei permutations with the same cycle structure. When a R\'{e}dei permutation
has a unique cycle structure that is not shared by any other R\'{e}dei permutation,
we call it \emph{isolated}. We show that the only isolated
R\'{e}dei permutations are the isolated R\'{e}dei involutions. Moreover, all
our results can be transferred to bijections of the form $mx$ and
$x^{m}$ on certain domains.
\end{abstract}
\maketitle

\section{Introduction}
\label{sec1}

The cycle structure of a permutation captures important information about
its dynamics. In the case of permutations induced by polynomials over a
finite field, the cycle structures of only a few polynomials are known.
Among these polynomials are the monomials~\cite{MR240184}, Dickson polynomials~\cite{MR1159877},
R\'{e}dei functions~\cite{MR3384830}, Chebyshev polynomials~\cite{MR3911212},
certain $q$-linearized polynomials~\cite{MR959697,MR3911214}, polynomials
with low Carlitz rank~\cite{MR2435050}, polynomials of the form
$x+\gamma f(x)$ over $\mathbb{F}_{q^{n}}$ where
$f\colon \mathbb{F}_{q^{n}}\to \mathbb{F}_{q}$~\cite{MR4149368}, and those
ones induced by certain piecewise-affine permutations of
$\{1,\ldots ,n\}$~\cite{MR4081007}. Another related challenging problem
is to find families of polynomials with the same cycle structure. The literature
in this direction is even more restrict; see~\cite{MR4149368} for instance.
In this paper we study permutations induced by R\'{e}dei functions over a
finite field with the same cycle structure.

Let ${\mathbb{F}}_{q}$ be the finite field of order $q$, and
$\mathbb{P}^{1}({\mathbb{F}}_{q})=\mathbb{F}_{q}\cup \{\infty \}$. We consider
the binomial expansion of $\displaystyle (x+\sqrt{y})^{m}$ written as
$N(x,y)+D(x,y)\sqrt{y}$. For $m\in \mathbb{N}$ and
$a \in {\mathbb{F}}_{q}$, the R\'{e}dei function
$R_{m,a}\colon \mathbb{P}^{1}(\mathbb{F}_{q}) \to \mathbb{P}^{1}(
\mathbb{F}_{q})$ is defined by
\begin{align*}
R_{m,a}(x)=
\begin{cases}
\dfrac{N(x,a)}{D(x,a)} & \text{ if } D(x,a)\neq 0 \text{ and } x\neq
\infty
\\
\infty & \text{ otherwise.}
\end{cases}
\end{align*}
When $a\in {\mathbb{F}}_{q}^{*}$ and $q$ is odd, Carlitz
\cite{MR137702} obtained the following explicit formula:
\begin{equation*}
\label{Rn}
R_{m,a}(x) = \sqrt{a}
\dfrac{(x+\sqrt{a})^{m}+(x-\sqrt{a})^{m}}{(x+\sqrt{a})^{m}-(x-\sqrt{a})^{m}},
\end{equation*}
where $\sqrt{a}$ is an element of $\mathbb{F}_{q}$ or
$\mathbb{F}_{q^{2}}$.

We refer to a R\'{e}dei bijection as a R\'{e}dei permutation. In~\cite{MR3384830}
Qureshi and Panario describe the cycle structure of a R\'{e}dei permutation
over $\mathbb{P}^{1}({\mathbb{F}}_{q})$. At the end of their paper they propose
the problem of investigating the conditions under which two R\'{e}dei permutations
$R_{m,a}$ and $R_{n,b}$ have the same cycle structure. The goal of this
paper is to address this question. Our first main result provides a complete
characterization of all pairs $(m,n)\in \mathbb{N}^{2}$ that satisfy the
following condition: $R_{m,a}$ and $R_{n,b}$ are R\'{e}dei permutations with
the same cycle structure for some $a,b\in \mathbb{F}_{q}$ with
$\chi (a)=\chi (b)$ and $q$ odd. The case $\chi (a)\neq \chi (b)$ will
be addressed in a forthcoming paper. By $\chi (a)$, we mean the quadratic
character of $a$, that is, $\chi (a) = 1$ if $a$ is a square in
${\mathbb{F}}_{q}$, and $\chi (a)=-1$ otherwise.

In our recent work~\cite{MR4171318} we investigate R\'{e}dei permutations
that only decompose into $1$- and $p$-cycles, where $p$ is $4$ or a fixed
prime. We note that our results in~\cite{MR4171318} do not immediately
provide a way of characterizing all R\'{e}dei permutations that have the
same cycle structure given by $1$- and $p$-cycles, except when $p=2$. In
fact, our approach in the present paper is different from the one we used
there.

A closer investigation of the pairs $(m,n)$ that yield R\'{e}dei permutations
with the same cycle structure reveals some symmetric properties. We explore
them, and use an example to illustrate the distribution of such pairs over
certain lines, making those properties evident. We also find several explicit
families of R\'{e}dei permutations with the same cycle structure.

When a R\'{e}dei permutation has a unique cycle structure, in the sense that
no other R\'{e}dei permutation has the same cycle structure, we call it \emph{isolated}.
Our second main result shows that the only isolated R\'{e}dei
permutations are the isolated R\'{e}dei involutions. This provides a connection
with our previous work. Using results from~\cite{MR4171318}, we determine
all values of $m$ for which $R_{m,a}$ is an isolated involution.

We organize this paper as follows. In Section~\ref{preliminaries} we present
some results that are instrumental in our proofs. Section~\ref{principal}
contains our two main results. In Section~\ref{meio} we show several results
that guarantee the existence of certain pairs $(m,n)$ such that
$R_{m,a}$ and $R_{n,b}$ have the same cycle structure, and indicate how
some of them are related to each other through some symmetries. In Section~\ref{families}
we exhibit several families of R\'{e}dei permutations with the same cycle
structure. In Section~\ref{conclusao} we use a link between R\'{e}dei functions
$R_{m,a}$ and functions of the form $mx$ and $x^{m}$ to transfer the results
in this paper to these two functions over certain domains.

From this point on, unless specified otherwise, we will assume that
$q$ is a power of an odd prime, $a,b\in \mathbb{F}_{q}^{*}$,
$\chi \in \{-1,1\}$, and all other variables are positive integers.

\section{The cycle structure of a R\'edei permutation}
\label{preliminaries}

In this section we discuss some known results about R\'{e}dei permutations
and their cycle structures. In~\cite{bibVTEX} Deng obtains the following
condition for arbitrary permutations on finite sets to have the same cycle
structure.

\begin{proposition}%
\label{Deng}%
\cite[Lemma 9]{bibVTEX} Let $X_{1}$ and $X_{2}$ be finite sets, and
$f_{1}\colon X_{1}\to X_{1}$ and $f_{2}\colon X_{2}\to X_{2}$ be permutations.
Then $f_{1}$ and $f_{2}$ have the same cycle structure if and only if the
numbers of fixed points of $f_{1}^{r}$ and $f_{2}^{r}$ are the same for
every positive integer $r$.
\end{proposition}

In~\cite{MR3384830} Qureshi and Panario describe precisely when a R\'{e}dei
function induces a permutation on $\mathbb{P}^{1}(\mathbb{F}_{q})$. They
also give the cycle structure of a R\'{e}dei permutation and a formula for
the number of fixed points. We denote the order of $m$ modulo $d$ by
$o_{d}(m)$, and the Euler function by $\varphi $. A cycle decomposition
consisting of $n_{1}$ fixed points and $n_{c_{i}}$ directed $c_{i}$-cycles
for $i=1,\dots ,k$ is denoted by
$\left (n_{1} \times \{\bullet \}\right ) \oplus \bigoplus _{i=1}^{k}
\left (n_{c_{i}} \times \mathrm{Cyc}(c_{i})\right )$ where
$\mathrm{Cyc(c)}$ denotes a directed $c$-cycle.

\begin{proposition}%
\label{fixpt}
%
%
%
\begin{enumerate}[{\normalfont (i)}]
\item
\cite[Corollary 4.7]{MR3384830} The R\'{e}dei function $R_{m,a}$ induces
a permutation on $\mathbb{P}^{1}(\mathbb{F}_{q})$ if and only if
$\gcd (m,q-\chi (a))=1$. In this case, the decomposition of
$R_{m,a}$ in disjoint cycles is
\begin{equation*}
\left ((1+\chi (a)) \times \{\bullet \}\right ) \oplus \bigoplus _{d
\mid q-\chi (a)} \left \{  \dfrac{\varphi (d)}{o_{d}(m)}\times
\mathrm{Cyc}(o_{d}(m))\right \}  .%
\end{equation*}
\item
\cite[Corollary 4.8]{MR3384830} The number of fixed points of
$R_{m,a}$ in $\mathbb{P}^{1}(\mathbb{F}_{q})$ is
$\gcd (m-1,q-\chi (a))+\chi (a)+1$.
\item
\cite[Corollary 4.9]{MR3384830} The number of fixed points in the
$r^{\mathrm{th}}$ iteration of $R_{m,a}$ is
$\gcd (m^{r}-1,q-\chi (a))+\chi (a)+1$.
\end{enumerate}
\end{proposition}

By combining the above results, we have the following criterion to determine
whether or not two R\'{e}dei permutations have the same cycle structure.

\begin{proposition}%
\label{conseq1}
Suppose $\gcd (m,q-\chi (a)) = \gcd (n,q-\chi (b))=1$. Then
$R_{m,a}$ and $R_{n,b}$ have the same cycle structure if and only if
$\gcd (m^{r}-1,q-\chi (a))+\chi (a) = \gcd (n^{r}-1,q-\chi (b))+\chi (b)$
for all positive integers $r$.
\end{proposition}

\section{Main results}
\label{principal}

We recall that $R_{m,a}$ permutes $\mathbb{P}^{1}(\mathbb{F}_{q})$ if and
only if $\gcd (m,q-\chi (a))=1$. Since $q$ is odd, it follows that
$m$ must be odd. In addition, $R_{m,a}$ and $R_{n,a}$ induce the same function
if and only if $m\equiv n \pmod{q-\chi (a)}$. The permutations
$R_{m,a}$ and $R_{m,b}$ have the same cycle structure whenever
$\chi (a)=\chi (b)$, so to ease the notation, we use $\chi $ to represent
this common value, that is, $\chi \in \{-1,1\}$.

Let $\mathcal{S}_{\chi }^{q}$ be the set of all
$(m,n)\in \mathbb{N}^{2}$ such that $R_{m,a}$ and $R_{n,b}$ are R\'{e}dei
permutations with the same cycle structure for some
$a,b\in {\mathbb{F}}_{q}$ with $\chi (a)=\chi (b)=\chi $. Clearly,
$(m,n)\in S_{\chi }^{q}$ if and only if $(n,m)\in S_{\chi }^{q}$. In this
section we give necessary and sufficient conditions for a pair
$(m,n)$ to belong to $\mathcal{S}_{\chi }^{q}$. We observe that
$R_{1,a}$ is the identity function, thus its cycle structure consists of
isolated points only. Therefore we may restrict ourselves to the case
$1<m<n<q-\chi $.

We start with some auxiliary results that will allow us to prove our first
main result. The $p$-adic valuation of an integer $z$ is denoted by
$\nu _{p}(z)= \alpha $ where $p^{\alpha }\mid \mid z$.

\begin{lemma}%
\label{lemmagcdp}%
Let $p$ be a prime such that $p\nmid m$, and let $\theta =o_{p}(m)$.
%
%
%
\begin{enumerate}[{\normalfont (i)}]
\item If $p$ is odd, or if $p=2$ and $\nu _{2}(m-1)>1$, then
\begin{equation*}
\nu _{p}(m^{r}-1) =
\begin{cases}
0& \text{ if }\theta \nmid r
\\
\nu _{p}(m^{\theta }-1)+\nu _{p}(t) &\text{ if }r=t\theta .
\end{cases}
\end{equation*}
\item If $p=2$, then
\begin{equation*}
\nu _{2}(m^{r}-1) =
\begin{cases}
\nu _{2}(m-1)& \text{ if } r \text{ is odd}
\\
\nu _{2}(m^{2}-1)+\nu _{2}(r)-1 & \text{ if }r \text{ is even}.
\end{cases}
\end{equation*}
\end{enumerate}
\end{lemma}

\begin{proof}
The result follows from the Lifting the Exponent Lemma.
\end{proof}

\begin{lemma}%
\label{mainprimelemma}
Let $p$ be a prime such that $p\nmid m,n$, and $\alpha $ be a positive
integer. Then
$\gcd (m^{r}-1,p^{\alpha }) = \gcd (n^{r}-1,p^{\alpha })$ for any positive
integer $r$ if and only if $\theta := o_{p}(m) = o_{p}(n)$ and one of the
following conditions holds:
%
%
%
\begin{enumerate}[{\normalfont (i)}]
\item $\alpha =1$,
\item $\alpha >1$, $p\neq 2$, and
$\gcd (m^{\theta }-1,p^{\alpha }) = \gcd (n^{\theta }-1,p^{\alpha })$,
\item $\alpha >1$, $p = 2$, $\nu _{2}(m-1)>1$, and
$\gcd (m-1,2^{\alpha }) = \gcd (n-1,2^{\alpha })$,
\item $\alpha >1$, $p = 2$, $\nu _{2}(m-1)=1$, and
$\gcd (m^{j}-1,2^{\alpha }) = \gcd (n^{j}-1,2^{\alpha })$ for
$j \in \{1,2\}$.
\end{enumerate}
\end{lemma}

\begin{proof}
The result follows from {Lemma~\ref{lemmagcdp}}.
\end{proof}

We are now ready to state our first main result.

\begin{theorem}%
\label{mainthm}
Suppose $q-\chi =p_{1}^{\alpha _{1}}\cdots p_{t}^{\alpha _{t}}$ is the
prime factorization of $q-\chi $, $m$ is coprime with $q-\chi $, and
$\theta _{i} = o_{p_{i}}(m)$. Then $(m,n) \in S^{q}_{\chi }$ if and only
if $n= m+k(q-\chi )/d$, where $d$ is a proper divisor of $q-\chi $ and
$k$ is an integer, and for each $p_{i}$ that divides $d$ the following
conditions hold:
%
%
%
\begin{enumerate}[{\normalfont (i)}]
\item $p_{i}\nmid n$ and $o_{p_{i}}(n) = \theta _{i}$,
\item $\gcd (m^{\theta _{i}}-1,p_{i}^{\alpha _{i}}) = \gcd (n^{\theta _{i}}-1,p_{i}^{
\alpha _{i}})$.
\item If $p_{i} = 2$, $\alpha _{i}>1$ and $\nu _{2}(m-1)=1$, then
$\gcd (m^{2}-1,2^{\alpha _{i}}) = \gcd (n^{2}-1,2^{\alpha _{i}})$.
\end{enumerate}
\end{theorem}
\begin{proof}
Suppose $(m,n) \in S_{\chi }^{q}$. Then $\gcd (n,q-\chi )=1$, and in particular
$p_{i} \nmid n$. Write $\gcd (n-m, q-\chi ) = (q-\chi )/d$. Since
$n-m = (n-1)-(m-1)$, it follows that $n-m$ is divisible by
$\gcd (m-1,q-\chi ) = \gcd (n-1,q-\chi ) \geq 2$, hence $d$ is a proper
divisor of $q-\chi $. Thus $n = m + k(q-\chi )/d$ for some integer
$k$. Since $\gcd (m^{r}-1,q-\chi ) = \gcd (n^{r}-1,q-\chi )$ for any positive
integer $r$, we obtain that
$\gcd (m^{r}-1,p_{i}^{\alpha _{i}}) = \gcd (n^{r}-1,p_{i}^{\alpha _{i}})$,
hence the remaining conditions must hold.

Conversely, assume $n=m+k(q-\chi )/d$ with $d$ satisfying the conditions.
We need to show that $\gcd (m^{r}-1,q-\chi )=\gcd (n^{r}-1,q-\chi )$ for
any positive integer $r$. If $p_{i}\nmid d$, then
$p_{i}^{\alpha _{i}} \mid k(q-\chi )/d$, which implies that
$\gcd (m^{r}-1,p_{i}^{\alpha _{i}}) = \gcd (n^{r}-1,p_{i}^{\alpha _{i}})$.
Since
\begin{equation*}
\gcd (u,q-\chi ) = \prod _{i=1}^{t} \gcd (u,p_{i}^{\alpha _{i}}),%
\end{equation*}
it remains to show that
\begin{equation*}
\label{gcdpalpha2}
\gcd (m^{r}-1,p_{i}^{\alpha _{i}}) = \gcd ( n^{r} -1, p_{i}^{\alpha _{i}})%
\end{equation*}
for any positive integer $r$ and each $i$ such that $p_{i} \mid d$, which
follows from {Lemma~\ref{mainprimelemma}}.
\end{proof}

\begin{remark}%
\label{remark2}
When $\alpha _{i}=1$, Condition (ii) is always satisfied and Condition
(iii) does not apply. When $\nu _{p_{i}}(d)< \alpha _{i}$, Condition (i)
is always satisfied.
\end{remark}

\begin{remark}
\label{rem3.5}
When the divisor $d$ in {Theorem~\ref{mainthm}} is $1$, it follows that
$m \equiv n \pmod{q-\chi }$, hence $R_{m,a}$ and $R_{n,a}$ induce the same
permutation.
\end{remark}

\begin{remark}%
\label{remark1}
Assume $1<m<n<q-\chi $. Then $0<n-m<q-\chi $. By taking a proper divisor
$d$ of $q-\chi $ such that $\gcd (n-m,q-\chi ) = (q-\chi )/d$, it follows
that $n = m + k(q-\chi )/d$ with $0<k<d$ and $\gcd (k,d)=1$.
\end{remark}

\begin{remark}%
\label{remark3}
When $\nu _{2}(q-\chi )=1$, we observe that
$\gcd (m+(q-\chi )/2,q-\chi )\neq 1$, so $R_{m + (q-\chi )/2,a}$ does not
permute $\mathbb{P}^{1}(\mathbb{F}_{q})$ for any $a\in \mathbb{F}_{q}$ with
$\chi (a)=\chi $. Suppose $\alpha =\nu _{2}(q-\chi )>1$ and
$\gcd (m,q-\chi )=1$. Then $(m,m + (q-\chi )/2)\in S_{\chi }^{q}$ if and
only if $m \not\equiv 1 \pmod{2^{\alpha -1}}$.
\end{remark}

\begin{example}%
\label{49}
Let $q=49$. We apply {Theorem~\ref{mainthm}} to find all R\'{e}dei permutations
with the same cycle structures for a fixed quadratic character
$\chi $. We only consider $R_{m,a}$ with $1<m<q-\chi $.

First, let $\chi =-1$. Then $q-\chi =50 = 2\cdot 5^{2}$. There are
$\varphi (50) = 20$ values of $m$ for which $R_{m,a}$ is a permutation.
Disregarding $m=1$, which gives the identity map, those values are
$3$, $7$, $9$, $11$, $13$, $17$, $19$, $21$, $23$, $27$, $29$, $31$,
$33$, $37$, $39$, $41$, $43$, $47$, and $49$. By {Remark~\ref{remark1}},
for each $m$ we consider $m+k\cdot 50/d$, with $d \in \{2,5,10,25\}$,
$0<k<d$ and $\gcd (k,d)=1$. Note that $50/d$ is odd for any even divisor
$d$, hence $m+k\cdot 50/d$ is even, and so it does not yield a permutation.
We observe that in this case Condition (i) in {Theorem~\ref{mainthm}} is
violated by any such $d$. Thus, it is enough to consider
$d \in \{5,25\}$.

We start with $m=3$. We have $o_{5}(3)=4$ and $\gcd (3^{4}-1,25)=5$. When
$d=5$, {Remark~\ref{remark2}} implies that we only need to check Condition
(ii) with $p_{i}=5$. When $d=25$, we check Conditions (i) and (ii) for
$p_{i}=5$. The calculations corresponding to each possible value of
$k$ for $d=5$ and $25$ are in {Tables~\ref{tabela1exemplo} and \ref{tabela2exemplo}}, respectively. The result is that the permutations
$R_{m,a}$ with $m\in \{3,13,17,23,27,33,37,47\}$ and $\chi (a)=-1$ have
the same cycle structure.

\begin{table}
 \begin{tabular}{|c c c||c c c|} 
 \hline
 $k$ & $n$ & $\gcd(n^4-1,25)$ & $k$ & $n$ & $\gcd(n^4-1,25)$\\
 \hline\hline
 1 & 13 & 5 &  3 & 33 & 5 \\ 
 \hline
 2 & 23 & 5 & 4 & 43 & 25\\
 \hline
\end{tabular}
\caption{Calculations for $d=5$ and each possible value of $k$.}\label{tabela1exemplo}
\end{table}

\begin{table}
 \begin{tabular}{|c c c c c||c c c c c|} 
 \hline
 $k$ & $n$ & $5\nmid n?$ & $o_5(n)$ & $\gcd(n^4-1,25)$ & $k$ & $n$ & $5\nmid n?$ & $o_5(n)$ & $\gcd(n^4-1,25)$\\
 \hline\hline
 1 & 5 & no & & & 13 & 29 & yes & 2 & \\ 
 \hline
 2 & 7 & yes & 4 & 25 & 14 & 31 & yes & 1 &  \\
 \hline
 3 & 9 & yes & 2 & &  16 & 35 & no &  & \\
 \hline
 4 & 11 & yes & 1 & & 17 & 37 & yes & 4 & 5\\
 \hline
 6 & 15 & no &  & & 18 & 39 & yes & 2 & \\
 \hline
 7 & 17 & yes & 4 & 5 & 19 & 41 & yes & 1 &\\
 \hline
 8 & 19 & yes & 2 & & 21 & 45 & no &  & \\
 \hline
 9 & 21 & yes & 1 & & 22 & 47 & yes & 4 & 5\\
 \hline
 11 & 25 & no &  & & 23 & 49 & yes & 2 & \\
 \hline
 12 & 27 & yes & 4 & 5 & 24 & 1 & yes & 1 & \\
 \hline
\end{tabular}
\caption{Calculations for $d=25$ and each possible value of $k$.}\label{tabela2exemplo}
\end{table}

Next we consider $m=7$, with $o_{5}(7)=4$. Since the values of
$7+k\cdot 50/d$ for $d=5$ are $17,27,37$, and $47$, and the permutations
given by these values have the same cycle structure than the one induced
by $m=3$, we can skip $d=5$. For $d=25$, the only value of $n$ for which
Conditions (i) and (ii) are satisfied is $n=43$ with $k=18$. Since
$o_{5}(7)=o_{5}(3)$, no further calculations are necessary, as the orders
and the gcds are in the tables.

For $m=9$, we have $o_{5}(9)=2$ and $\gcd (9^{2}-1,25)=5$, and we first
consider $d=5$. The values of $n=9+k\cdot 50/d$ are $19,29,39$, and
$49$, and we have $\gcd (n^{2}-1,25)=5$ for $n=19,29,39$, and
$\gcd (49^{2}-1,25)=25$. We observe that these are the only values of
$n$ for which the order modulo $5$ is $2$, thus we may disregard
$d=25$, and we conclude that the R\'{e}dei permutations with
$m \in \{9,19,29,39\}$ and $\chi =-1$ have the same cycle structure.

For $m=11$, we have $o_{5}(11)=1$ and $\gcd (11-1,25) =5$. Similarly to
the case $m=9$, for $d=5$ we obtain three values of $n$ for which the cycle
structure is the same as $R_{11,a}$, namely $n\in \{21,31,41\}$, and
$d=25$ gives no other value.

Having exhausted all the possibilities, we conclude that there are five
distinct cycle structures among the nontrivial R\'{e}dei permutations over
$\mathbb{P}^{1}(\mathbb{F}_{49})$ with $\chi =-1$, given by $m$ in the
following sets: $\{3,13,17,23,27,33,37,47\}$, $\{7,43\}$,
$\{9,19,29,39\}$, $\{11,21,31,41\}$, and $\{49\}$.

Now we consider $\chi =1$. In this case, $q-\chi =48=2^{4}\cdot 3$. The
values of $m$ for which $R_{m,a}$ is a nontrivial permutation are
$5$, $7$, $11$, $13$, $17$, $19$, $23$, $25$, $29$, $31$, $35$, $37$,
$41$, $43$, and $47$. For $d=2$, by {Remark~\ref{remark3}} we have
$(m,m+24) \in S^{49}_{1}$ if and only if $m \not\equiv 1 \pmod{8}$. Thus
$(5,29)$, $(7,31)$, $(11,35)$, $(13,37)$, $(19,43)$, and
$(23,47) \in S^{49}_{1}$. The other divisors of $q-\chi =48$ do not produce
any other pair. Therefore there are nine distinct cycle structures among
the nontrivial R\'{e}dei permutations with $\chi =1$.

We summarize our findings in {Table~\ref{tabela2ciclo}} by listing all R\'{e}dei
permutations over $\mathbb{P}^{1}(\mathbb{F}_{49})$ along with their cycle
structures. Note that the permutations with a unique cycle structure only
have cycles of length $1$ or $2$. This is no coincidence, as we show later
in this section ({Corollary~\ref{coroisolated}}).

We observe that for $q=47$ and $\chi =-1$ the values of $m$ for which
$R_{m,a}$ is a permutation and those with identical cycle structure are
the same as for $q=49$ and $\chi =1$, since $q-\chi =48$ in both cases.
In other words, $S_{-1}^{47}=S_{1}^{49}$. However, the permutations on
$\mathbb{P}^{1}(\mathbb{F}_{47})$ have two fixed points less than the corresponding
ones on $\mathbb{P}^{1}(\mathbb{F}_{49})$.\quad$\diamond $

\begin{table}
 \begin{tabular}{|l|l| l |} 
 \hline
 $\chi(a)$ & $m$ & Cycle Structure\\
 \hline\hline
$-1$ & $1$ &  $50 \times \{\bullet\}$ \\ 
\hline
$-1$ & $3,13,17,23,27,33,37,47$ &  $\left(2 \times\{\bullet\}\right) \oplus \left(2\times \mathrm{Cyc}(4)\right) \oplus \left(2\times  \mathrm{Cyc}(20)\right)$\\ 
\hline
$-1$ & $7,43$ &  $\left(2 \times \{\bullet\}\right) \oplus \left(12\times \mathrm{Cyc}(4)\right)$\\ 
\hline
$-1$ & $9,19,29,39$ & $ \left(2 \times \{\bullet\}\right) \oplus \left(4\times \mathrm{Cyc}(2)\right) \oplus \left(4\times  \mathrm{Cyc}(10)\right)$\\ 
\hline
$-1$ & $11,21,31,41$ &  $\left(10 \times \{\bullet\}\right) \oplus \left(8\times \mathrm{Cyc}(5)\right)$\\ 
\hline
$-1$ & $49$ &  $\left(2 \times \{\bullet\}\right) \oplus \left(24\times \mathrm{Cyc}(2)\right)$\\ 
\hline
$1$ & $1$ &  $50 \times \{\bullet\}$ \\ 
\hline
$1$ & $5,29$ &  $\left(6 \times \{\bullet\}\right) \oplus \left(10\times \mathrm{Cyc}(2)\right) \oplus \left(6\times \mathrm{Cyc}(4)\right)$\\
\hline
$1$ & $7,31$ & $\left(8 \times \{\bullet\}\right) \oplus \left(21\times \mathrm{Cyc}(2)\right)$\\
\hline
$1$ & $11,35$ & $\left(4 \times \{\bullet\}\right) \oplus \left(11\times \mathrm{Cyc}(2)\right) \oplus \left(6\times \mathrm{Cyc}(4)\right)$\\
\hline
$1$ & $13,37$& $\left(14 \times \{\bullet\}\right) \oplus \left(6\times \mathrm{Cyc}(2)\right) \oplus \left(6\times \mathrm{Cyc}(4)\right)$\\
\hline
$1$ & $17$& $\left(18 \times \{\bullet\}\right) \oplus \left(16\times \mathrm{Cyc}(2)\right)$\\
\hline
$1$ & $19,43$ &$\left(8 \times \{\bullet\}\right) \oplus \left(9\times \mathrm{Cyc}(2)\right) \oplus \left(6\times \mathrm{Cyc}(4)\right)$\\
\hline
$1$ & $23,47$& $\left(4 \times \{\bullet\}\right) \oplus \left(23\times \mathrm{Cyc}(2)\right)$\\
\hline
$1$ & $25$& $\left(26 \times \{\bullet\}\right) \oplus \left(12\times \mathrm{Cyc}(2)\right)$\\
\hline
$1$ & $41$& $\left(10 \times \{\bullet\}\right) \oplus \left(20\times \mathrm{Cyc}(2)\right)$\\
\hline
 \end{tabular}
 \caption{All R\'edei permutations $R_{m,a}$ over $\mathbb{P}^1\left(\mathbb F_{49}\right)$, with  $1\leq m <49-\chi(a)$, grouped by their cycle structures.}\label{tabela2ciclo}
 \end{table}
\end{example}

The next result shows that if $\rho :=o_{q-\chi }(m) > 2$, then the pair
$(m,m^{\ell})$ is in $S_{\chi}^q$ for any $\ell$ coprime with $\rho$. We observe that if $\rho>2$ and $1< \ell<\rho$, then $m^{\ell} \not\equiv m \pmod{q-\chi}$, thus $R_{m,a}$ and $R_{m^{\ell},a}$ are distinct R\'{e}dei permutations with the same cycle structure. In particular, $R_{m^{\rho-1},a}$ is the inverse of $R_{m,a}$.

\begin{proposition}%
\label{isolated}
If $\rho =o_{q-\chi }(m)>2$, then $(m,m^{\ell}) \in S_{\chi}^q$ for any integer $\ell$ coprime with $\rho$.
\end{proposition}

\begin{proof}
Let $p$ be a prime divisor of $q-\chi $ and
$\alpha =\nu _{p}(q-\chi )$, and let $\theta =o_{p}(m)$. Since
$\theta $ and $\ell$ are coprime, it follows that
$o_{p}(m^{\ell}) = \theta $. By {Lemma~\ref{lemmagcdp}}, it suffices to
show that
$\gcd ((m^{\ell})^{r}-1, p^{\alpha }) = \gcd (m^{r}-1, p^{\alpha })$ for
$r=\theta $; in addition, when $p=2$, the identity also needs to be verified
for $r=2$.

Let $\beta = \nu _{p}(m^{\theta }-1)$. Write
$m^{\theta } -1=kp^{\beta }$, for some integer $k$ with $p\nmid k$. Then
\begin{align*}
(m^{\ell})^{\theta } & = (1+kp^{\beta })^{\ell}
\\
& = \sum _{i=0}^{\ell}\binom{\ell}{i} k^{i}p^{i\beta }
\\
& = 1+p^{\beta }\left ( \ell k+ \sum _{i=2}^{\ell}
\binom{\ell}{i}k^{i}p^{(i-1)\beta } \right ).
\end{align*}
Clearly, $p$ divides the summation in the last expression. If
$p \nmid \ell$, then
$\nu _{p}((m^{\ell})^{\theta }-1) = \beta $. If $p\mid \ell$, then
$p\nmid \rho $, which implies that $o_{p^{\alpha }}(m)=\theta $. Hence
$\alpha \leq \beta \leq \nu _{p}((m^{\ell})^{\theta }-1)$. In both cases
we conclude that
$\gcd ((m^{\ell})^{\theta }-1, p^{\alpha }) = \gcd (m^{\theta }-1, p^{
\alpha })$. When $p=2$, the same reasoning shows that
$\gcd ((m^{\ell})^{2}-1, p^{\alpha }) = \gcd (m^{2}-1, p^{\alpha })$.
\end{proof}

We say that a R\'{e}dei permutation $R_{m,a}$ is \emph{isolated} if
$m \equiv n \pmod{q-\chi }$ for all pairs $(m,n) \in S_{\chi }^{q}$, that
is, there is no distinct R\'{e}dei permutation with the same cycle structure.
{Proposition~\ref{isolated}} shows that if $R_{m,a}$ is isolated, then
$o_{q-\chi }(m)$ is either $1$ or $2$, in which cases $R_{m,a}$ is the identity
map or an involution, respectively. An involution is a permutation that
only decomposes into cycles of length $1$ or $2$. It follows that
$R_{m,a}$ is an involution if and only if
$m^{2}\equiv 1 \pmod{q-\chi }$. This shows our second main result.

\begin{corollary}%
\label{coroisolated}
The only isolated R\'{e}dei permutations are the isolated R\'{e}dei involutions.
\end{corollary}

If $R_{m,a}$ is an involution, then it may or may not be isolated. In
\cite{MR4171318} the authors give a complete description of R\'{e}dei involutions
and their cycle structures. In this case, a cycle structure is completely
determined by the number of fixed points. It turns out that a R\'{e}dei involution
either is isolated or has a cycle structure that is shared by exactly one
other R\'{e}dei involution. The next result follows from Propositions 4 and
5 in~\cite{MR4171318}. It establishes when a R\'{e}dei involution with
a prescribed cycle structure exists, and if it does, it gives the values
of $m$ explicitly.

\begin{proposition}%
\label{involutions}
Let
$q-\chi =2^{\alpha _{0}}p_{1}^{\alpha _{1}}\cdots p_{r}^{\alpha _{r}}$
be the prime factorization of $q-\chi $, and
$d=2^{\beta _{0}}p_{1}^{\beta _{1}}\cdots p_{r}^{\beta _{r}}$ be a proper
divisor of $q-\chi $. Then there exists a R\'{e}dei involution
$R_{m,a}$ with $d+\chi +1$ fixed points over
$\mathbb{P}^{1}(\mathbb{F}_{q})$ if and only if
$\beta _{i}\in \{0,\alpha _{i}\}$ for $1\leq i\leq r$ and one of the following
situations occurs:
%
%
%
\begin{enumerate}[{\normalfont (i)}]
\item $\beta _{0}\in \{\alpha _{0}-1,\alpha _{0}\}$ and
$\beta _{0}\geq 1$. In this case, $R_{m,a}$ is isolated and
$m \equiv k(q-\chi )/d-1 \pmod{q-\chi }$ for
\begin{equation*}
k =
\begin{cases}
\left ( \dfrac{q-\chi }{2d} \right )^{\varphi (d)-1} + \dfrac{d}{2} &
\text{ if } \beta _{0}=\alpha _{0}-1
\\
2\left ( \dfrac{q-\chi }{d} \right )^{\varphi (d)-1} & \text{ if }
\beta _{0}=\alpha _{0}
\end{cases}
\end{equation*}
with $k$ reduced modulo $d$.
\item $\alpha _{0}\geq 3$ and $\beta _{0}=1$. In this case,
$(m,n)\in S_{\chi }^{q}$ where
$m\text{ or }n \equiv k(q-\chi )/d-1 \pmod{q-\chi }$ for
\begin{equation*}
k =\left (\dfrac{q-\chi }{2d}\right )^{\varphi (d)-1}
\end{equation*}
with $k$ reduced modulo $d$, and
$m\equiv n+(q-\chi )/2 \pmod{q-\chi }$. Moreover, if
$(m,\ell )\in S_{\chi }^{q}$, then $\ell \equiv m$ or
$n\pmod{q-\chi }$.
\end{enumerate}
\end{proposition}

As a consequence of {Corollary~\ref{coroisolated}} and {Proposition~\ref{involutions}}
we have the following.

\begin{proposition}
\label{prop3.12}
Let
$q-\chi =2^{\alpha _{0}}p_{1}^{\alpha _{1}}\cdots p_{r}^{\alpha _{r}}$
be the prime factorization of $q-\chi $ and $M$ be the number of isolated
R\'{e}dei permutations over $\mathbb{P}^{1}(\mathbb{F}_{q})$. Then
\begin{equation*}
M=
\begin{cases}
2^{r}&\text{ if }\alpha _{0}=1
\\
2^{r+1}&\text{ otherwise. }
\end{cases}
\end{equation*}
\end{proposition}

\section{Symmetries in $S_{\chi }^{q}$}
\label{meio}

In this section we present alternative ways for generating pairs in
$S_{\chi }^{q}$. We also show the relationship between several pairs in
this set. Some results guarantee the existence of a pair based on the existence
of another pair. These results shed light on the structure of
$S_{\chi }^{q}$, particularly on the distribution of the pairs, as we show
in {Example~\ref{continuacao}}.

\begin{proposition}%
\label{propsym}
Suppose $q-\chi = p_{1}^{\alpha _{1}}\cdots p_{r}^{\alpha _{r}}$ is the
prime factorization of $q-\chi $ and $d$ is a proper divisor of
$q-\chi $. If $(m,n) \in S^{q}_{\chi }$, then
$(m+k(q-\chi )/d,n+k(q-\chi )/d) \in S^{q}_{\chi }$ if and only if for each
prime divisor $p_{i}$ of $d$ the following conditions hold:
%
%
%
\begin{enumerate}[{\normalfont (i)}]
\item $p_{i}\nmid m+k(q-\chi )/d, n+k(q-\chi )/d$ and
$\theta _{i}:= o_{p_{i}}(m+k(q-\chi )/d) = o_{p_{i}}(n+k(q-\chi )/d)$,
\item $\gcd ((m+k(q-\chi )/d)^{\theta _{i}}-1,p_{i}^{\alpha _{i}}) = \gcd ((n+k(q-
\chi )/d)^{\theta _{i}} -1,p_{i}^{\alpha _{i}})$.
\item If $p_{i} = 2$, $\alpha _{i}>1$ and $\nu _{2}((m+k(q-\chi )/d)-1)=1$, then
$\gcd ((m+k(q-\chi )/d)^{2} -1,2^{\alpha _{i}}) = \gcd ((n+k(q-\chi )/d)^{2}-1,2^{
\alpha _{i}})$.
\end{enumerate}
\end{proposition}

\begin{proof}
Since $\gcd (m^{r}-1,q-\chi ) = \gcd (n^{r}-1,q-\chi )$ for any positive
integer $r$, if $p_{i}\nmid d$ we have
$ \gcd ((m+k(q-\chi )/d)^{r}-1,p_{i}^{\alpha _{i}}) = \gcd (m^{r}-1,p_{i}^{
\alpha _{i}}) = \gcd (n^{r}-1,p_{i}^{\alpha _{i}}) = \gcd ((n+k(q-
\chi )/d)^{r}-1,p_{i}^{\alpha _{i}})$. If $p_{i} \mid d$, then by {Lemma~\ref{mainprimelemma}}
$\gcd ((m+k(q-\chi )/d)^{r}-1,p_{i}^{\alpha _{i}}) = \gcd ((n+k(q-
\chi )/d)^{r}-1,p_{i}^{\alpha _{i}})$ if and only if the conditions are
satisfied.
\end{proof}

The next result is a special case of {Proposition~\ref{propsym}} for
$d$ prime. We add some assumptions to obtain more explicit conditions.

\begin{corollary}%
\label{corosym2}
Let $p$ be a prime such that $\alpha = \nu _{p}(q-\chi ) >1$. If
$(m,n)\in S_{\chi }^{q}$ and $\nu _{p}(m-1)>0$, then
$(m + (q-\chi )/p,n + (q-\chi )/p) \in S_{\chi }^{q}$ if and only if one
of the following conditions holds:
%
%
%
\begin{enumerate}[{\normalfont (i)}]
\item $\nu _{p}(m-1) \neq \alpha -1$,
\item $ \dfrac{m-1}{p^{\alpha -1}}$ and
$ \dfrac{n-1}{p^{\alpha -1}}$ are either both or none congruent to
$-\dfrac{q-\chi }{p^{\alpha }}$ modulo $p$.
\end{enumerate}
\end{corollary}

\begin{proof}
Let $\beta =\nu _{p}(m-1)$. Since $\beta >0$, it follows that
$o_{p}(m)=1$. If $\beta \neq \alpha -1$, then
\begin{equation*}
\nu _{p}\left (m+\dfrac{q-\chi }{p}-1\right ) =
\begin{cases}
\beta & \text{ if } \beta <\alpha -1
\\
\alpha -1 & \text{ if } \beta >\alpha -1.
\end{cases}
\end{equation*}
The same identities hold for $\nu _{p}(n+(q-\chi )/p)$. When
$\beta = \alpha -1$, let $k = \dfrac{m-1}{p^{\alpha -1}}$,
$s=\dfrac{n-1}{p^{\alpha -1}}$ and
$\ell = \dfrac{q-\chi }{p^{\alpha }}$. Then
$m + (q-\chi )/p-1 = p^{\alpha -1}(k+\ell )$ and
$n + (q-\chi )/p-1 = p^{\alpha -1}(s+\ell )$. Thus the equality of the
gcds of {Proposition~\ref{propsym}}(ii) holds with $\theta =1$ if and
only if either $p$ divides both $k+\ell $ and $s+\ell $ or $p$ divides
neither one.
\end{proof}

The next corollary is yet another specialization of {Proposition~\ref{propsym}}
to $d=2$, or {Corollary~\ref{corosym2}} to $p=2$. As pointed out earlier,
we must have $\nu _{2}(q-\chi ) > 1$ and $\nu _{2}(m-1)>0$, thus the assumptions
of {Corollary~\ref{corosym2}} are not restrictive in this case. This result
allows us to limit the values of $m$ when searching for R\'{e}dei permutations
with the same cycle structure.

\begin{corollary}%
\label{corosym1}
Suppose $\nu _{2}(q-\chi )>1$. Then $(m,n)\in S_{\chi }^{q}$ if and only
if $(m + (q-\chi )/2,n + (q-\chi )/2)\in S_{\chi }^{q}$.
\end{corollary}
\begin{proof}
Let $\alpha = \nu _{2}(q-\chi )$ and $\beta = \nu _{2}(m-1)$. Then
\begin{equation*}
\nu _{2}\left (m+\dfrac{q-\chi }{2}-1\right )
\begin{cases}
=\beta & \text{ if } \beta <\alpha -1
\\
\geq \alpha & \text{ if } \beta =\alpha -1
\\
=\alpha -1 & \text{ if } \beta >\alpha -1.
\end{cases}
\end{equation*}
Furthermore, $\nu _{2}((m+(q-\chi )/2)^{2}-1) = \nu _{2}(m^{2}-1)$. Similar
results hold for $n$ and $n+(q-\chi )/2$. It follows that
$\gcd (m^{j}-1,2^{\alpha }) = \gcd (n^{j}-1,2^{\alpha })$ if and only if
$\gcd ((m+(q-\chi )/2)^{j}-1, 2^{\alpha }) = \gcd ((n+(q-\chi )/2)^{j}-1,
2^{\alpha })$.
\end{proof}

The following two propositions establish some sort of symmetry among the
pairs in $S_{\chi }^{q}$.

\begin{proposition}%
\label{comp}
Let $\alpha =\nu _{2}(q-\chi )$. If $(m,n) \in S_{\chi }^{q}$, then
$(q-\chi -m,q-\chi -n)\in S_{\chi }^{q}$ if and only if
$\gcd (m+1,2^{\alpha })=\gcd (n+1,2^{\alpha })$. In particular, if
$\alpha \in \{1,2\}$, then $(m,n)\in S_{\chi }^{q}$ if and only if
$(q-\chi -m,q-\chi -n)\in S_{\chi }^{q}$.
\end{proposition}
\begin{proof}
Let $m$ be an odd integer. If $r$ is even, then
%
\begin{equation}
\label{gcd1001}
\gcd ((q-\chi -m)^{r}-1,q-\chi ) = \gcd (m^{r}-1,q-\chi ).%
\end{equation}
If $r$ is odd, then we can write
%
\begin{align}
&\gcd ((q-\chi -m)^{r}-1,q-\chi )
\nonumber
\\
=& \gcd ((q-\chi -m)^{r}-1,(q-\chi )/2^{\alpha })\cdot \gcd ((q-\chi -m)^{r}-1,2^{
\alpha })
\nonumber
\\
= & \gcd (m^{r}+1, (q-\chi )/2^{\alpha }) \cdot \gcd (q-\chi -m-1,2^{
\alpha })
\label{gcdrodd}
\\
= & \gcd (m^{r}+1, (q-\chi )/2^{\alpha }) \cdot \gcd (m+1,2^{\alpha }),
\nonumber
\end{align}
where we applied {Lemma~\ref{lemmagcdp}}(ii) to obtain~{\eqref{gcdrodd}}.

Since $m^{2r}-1=(m^{r}-1)(m^{r}+1)$, $\gcd (m^{r}-1,m^{r}+1)=2$, and
$(q-\chi )/2^{\alpha }$ is odd, we have
\begin{equation*}
\gcd (m^{2r}-1,(q-\chi )/2^{\alpha })= \gcd (m^{r}-1,(q-\chi )/2^{
\alpha })\cdot \gcd (m^{r}+1,(q-\chi )/2^{\alpha }).%
\end{equation*}
Thus,
%
\begin{equation}
\label{gcd1000}
\gcd ((q-\chi -m)^{r}-1,q-\chi ) =
\dfrac{\gcd (m^{2r}-1,(q-\chi )/2^{\alpha })}{\gcd (m^{r}-1,(q-\chi )/2^{\alpha })}
\cdot \gcd (m+1,2^{\alpha }).%
\end{equation}

If $(m,n) \in S_{\chi }^{q}$, then
$\gcd (m^{r}-1,q-\chi ) = \gcd (n^{r}-1,q-\chi )$ for any positive integer
$r$. By~{\eqref{gcd1001}},
$\gcd ((q-\chi -m)^{r}, q-\chi ) = \gcd ((q-\chi -n)^{r}, q-\chi )$ for
any even $r$. By~{\eqref{gcd1000}}, if $r$ is odd, then
$\gcd ((q-\chi -m)^{r}, q-\chi ) = \gcd ((q-\chi -n)^{r}, q-\chi )$ if
and only if $\gcd (m+1,2^{\alpha })=\gcd (n+1,2^{\alpha })$.

In the case $\alpha \in \{1,2\}$, one can easily verify that
$\gcd (m-1,2^{\alpha }) = \gcd (n-1,2^{\alpha })$ if and only if
$\gcd (m+1,2^{\alpha }) = \gcd (n+1,2^{\alpha })$.
\end{proof}

By a similar argument one can also prove the following.

\begin{proposition}
\label{prop4.5}
Suppose $\alpha =\nu _{2}(q-\chi ) >1$ and $(m,n) \in S_{\chi }^{q}$. Then
$((q-\chi )/2-m, (q-\chi )/2-n) \in S_{\chi }^{q}$ if and only if
$\gcd (m+1,2^{\alpha }) = \gcd (n+1, 2^{\alpha })$. In particular, if
$\alpha = 2$, then $(m,n)\in S_{\chi }^{q}$ if and only if
$((q-\chi )/2-m,(q-\chi )/2-n)\in S_{\chi }^{q}$.
\end{proposition}

\begin{example}%
\label{continuacao}
We now return to the case $q=49$ to illustrate the results in this section.
All R\'{e}dei permutations over
$\mathbb{P}^{1}\left (\mathbb{F}_{49}\right )$ are given in {Table~\ref{tabela2ciclo}},
sorted out based on their cycle structures. The pairs
$(m,n) \in S_{\chi }^{49}$ for $\chi =-1$ and $\chi =1$ are shown in {Figs.~\ref{fig1}a
and \ref{fig1}b}, respectively. We only display $(m,n)$ with
$1<m< n<q-\chi $, symbol-coded by their cycle structures, that is, each
symbol identifies pairs whose coordinates yield R\'{e}dei permutations with
the same cycle structure. Clearly, points that belong to the same horizontal
or vertical line have the same symbol.

\begin{figure}[htbp]
\centering
\begin{subfigure}[b]{0.5\textwidth}
\begin{center}
\begin{tikzpicture}[scale=0.8]
  \begin{axis}
    \addplot[only marks, mark size=3pt] coordinates {
      (3, 13)
      (3, 17)
      (3, 23)
      (3, 27)
      (3, 33) 
      (3, 37) 
      (3, 47)
      (13, 17) 
      (13, 23) 
      (13, 27) 
      (13, 33) 
      (13, 37) 
      (13, 47) 
      (17, 23) 
      (17, 27) 
      (17, 33) 
      (17, 37) 
      (17, 47)
      (23, 27) 
      (23, 33) 
      (23, 37) 
      (23, 47) 
      (27, 33) 
      (27, 37) 
      (27, 47) 
      (29, 39) 
      (31, 41) 
      (33, 37) 
      (33, 47) 
      (37, 47) 
  };     
    \addplot[
                thick,
                mark=o,
                mark size=3pt,
                smooth, color=black] coordinates { 
      (7, 43) 
  };
    \addplot[only marks,thick, mark= square*, mark size=3pt, color = black ] coordinates { 
      (9, 19) 
      (9, 29) 
      (9, 39)
      (19, 29) 
      (19, 39)
  };
    \addplot[only marks,  mark= square, mark size=3pt, color = black] coordinates { 
      (11, 21) 
      (11, 31) 
      (11, 41)
      (21, 31) 
      (21, 41)
  };
   \end{axis}
\end{tikzpicture}
\end{center}
\subcaption{$\chi=-1$}\label{fig1partb}
\end{subfigure}%
\begin{subfigure}[b]{0.5\textwidth}
\begin{center}
\begin{tikzpicture}[scale=0.8]
  \begin{axis}
    \addplot[only marks,mark size=3pt] coordinates {
      (5, 29) 
  };     
    \addplot[only marks, thick, mark = o, mark size=3pt] coordinates { 
      (7, 31) 
  };
    \addplot[only marks, thick, mark = square*,mark size=3pt] coordinates { 
      (11, 35) 
  };
    \addplot[only marks, mark = square,mark size=3pt] coordinates { 
      (13, 37) 
  };
    \addplot[only marks, thick, mark = triangle*, mark size=4pt] coordinates { 
      (19, 43) 
  };
    \addplot[only marks, thick, mark=triangle, mark size=4pt] coordinates { 
      (23, 47) 
  };
   \end{axis}
\end{tikzpicture}
\end{center}
\subcaption{$\chi=1$}\label{fig1parta}
\end{subfigure}
\caption{Pairs $(m,n) \in S_{\chi}^{49}$ with $1<m< n<49-\chi$,  symbol-coded by their cycle structures.}\label{fig1}
\end{figure}
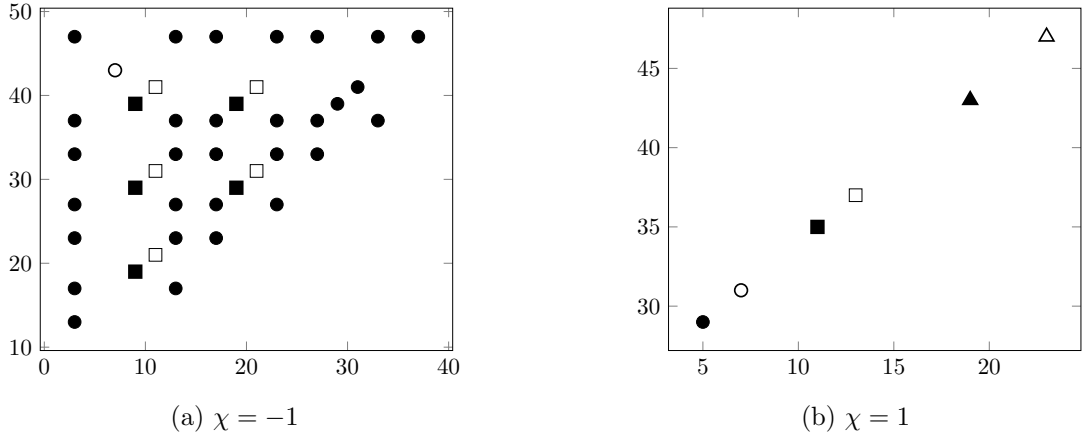

We notice that the pairs are distributed over lines of slope $1$, each
one of them containing all pairs $(m,n)$ for which
$n=m+k(q-\chi )/d$ for the same values of $d$ and $k$. The pairs along
these lines can be obtained from a single pair by applying {Proposition~\ref{propsym}}.
For instance, in {Fig.~\ref{fig1}}a the line at the bottom contains
the pairs $(13,17)$, $(23,27)$, and $(33,37)$. The pairs $(23,27)$ and
$(33,37)$ are obtained from $(13,17)$ by adding $10 = (q-\chi )/5$ and
$20 = 2(q-\chi )/5$ to both coordinates, respectively. In fact, all
$41$ pairs in {Fig.~\ref{fig1}}a can be generated from $(3,13)$,
$(3,17)$, $(3,23)$, $(3,27)$, $(3,33)$, $(3,37)$, $(3,47)$, and
$(13,17)$ in this fashion, for different values of $k(q-\chi )/d$. In {Table~\ref{tabela4}}
we show each one of these generator pairs together with the pairs they
generate along the lines. We observe that some of the pairs generated may
lie on another line, as it happens for instance with the pairs
$(17,23)$ and $(27,33)$. Starting with the pair $(3,47)$ on the line
$y=x+44$, adding $(q-\chi )/5$ and $2(q-\chi )/5$ to both coordinates and
reducing modulo $q-\chi =50$, would actually produce the pairs
$(23,17)$ and $(33,27)$, respectively. However, since we only display pairs
with the first coordinate smaller than the second one, the points shown
are $(17,23)$ and $(27,33)$, which lie on the line $y=x+6$. All points
in {Fig.~\ref{fig1}}b lie on the line $y=x+24$.

\begin{table}
 \begin{tabular}{|l|l|l|l|} 
 \hline
Equation & Generator  & $d$& $(m+k\cdot 50/d,n +k\cdot 50/d)$,\\
of the line & $(m,n)$ &  & $(n+k\cdot 50/d,m +k\cdot 50/d)$ in $\mathbb Z_{50}^2$\\
 \hline
 \hline
 $y=x+4$ & $(13,17)$ & $5$ & $(23,27)$, $(33,37)$  \\
\hline
 $y=x+44$ & $(3,47)$& &\\
 $y=x+6$ & & $5$  &$(17,23)$, $(27,33)$\\
\hline
 $y=x+10$ & $(3,13)$& $25$ &$(9,19)$, $(11,21)$, $(13,23)$, $(17,27)$, $(19,29)$, $(21,31)$, $(23,33)$,  \\
& & &  $(27,37)$, $(29,39)$, $(31,41)$,  $(37,47)$\\
\hline
 $y=x+14$  & $(3, 17)$& $5$ & $(13,27)$, $(23,37)$, $(33,47)$ \\
 $y=x+36$ & && $(7,43)$\\
\hline
 $y=x+34$ &$(3,37)$& $5$ &$(13,47)$\\
 $y=x+16$ & & & $(17,33)$\\
\hline
 $y=x+20$ &$(3,23)$& $25$ & $(9,29)$, $(11,31)$, $(13,33)$, $(17,37)$, $(19, 39)$, $(21,41)$,  $(27,47)$ \\
\hline
$y=x+24$ &(3,27)& $5$& $(13,37)$,  $(23,47)$\\ 
\hline
 $y=x+30$ & $(3,33)$& $25$ & $(9,39)$, $(11,41)$, $(17,47)$\\
\hline
 \end{tabular}
  \caption{The distribution of the points in Figure~\ref{fig1partb} over eleven lines and their corresponding generators.}\label{tabela4}
 \end{table}

Since $\nu _{2}(q-\chi ) =1$ for $\chi =-1$, {Proposition~\ref{comp}} implies
that the plot in {Fig.~\ref{fig1}}a is symmetric about the line
$y=-x+q-\chi $. This does not happen when $\chi =1$, as the reflections
of the pairs $(7,31)$ and $(13,47)$ about this line are not in
$S_{1}^{49}$.\quad$\diamond $
\end{example}

At the end of {Example~\ref{49}} we noted that
$S_{1}^{q}=S_{-1}^{\bar{q}}$ whenever $q-\bar{q}=2$. The next result shows
a relationship between $S_{\chi }^{q}$ and $S_{\chi }^{\bar{q}}$ when
$q$ and $\bar{q}$ satisfy another condition.

\begin{proposition}%
\label{propsymmetry1}
Suppose that $q$ and $\bar{q}$ are powers of a prime $p$ such that
$\nu _{p}(q-\chi ) = \nu _{p}(\bar{q}-\chi )=\alpha >0$ and
$(q-\chi )/p^{\alpha } + (\bar{q} -\chi )/p^{\alpha } \equiv 0 \pmod{p}$.
Then $(m,m + (q-\chi )/p)\in S_{\chi }^{q}$ if and only if
$(m,m - (\bar{q}-\chi )/p)\in S_{\chi }^{\bar{q}}$.
\end{proposition}

\begin{proof}
We observe that
$(q-\chi )/p^{\alpha } + (\bar{q} -\chi )/p^{\alpha } \equiv 0 \pmod{p}$ if
and only if
$(q-\chi )/p + (\bar{q} -\chi )/p \equiv 0 \pmod{p^{\alpha }}$. Then
\begin{equation*}
\gcd \left ( \left (m+\dfrac{q-\chi }{p}\right )^{r}-1,p^{\alpha }
\right ) = \gcd \left ( \left ( m - \dfrac{\bar{q}-\chi }{p} \right )^{r}-1,p^{
\alpha }\right ).%
\end{equation*}
Since
$\gcd ((m + (q-\chi )/p)^{r}-1,(q-\chi )/p^{\alpha }) = \gcd ((m - (
\bar{q}-\chi )/p)^{r}-1,(q-\chi )/p^{\alpha }) = \gcd (m^{r}-1,(q-
\chi )/p^{\alpha })$, the result follows.
\end{proof}

\section{Families of R\'edei permutations with the same cycle structure}
\label{families}

In this section we find explicit families of R\'{e}dei permutations that
share the same cycle structure. The first two families depend on the gcd
of integers of the form $c^{k}\pm 1$.

\begin{lemma}%
\label{mdc}
We have
%
%
%
\begin{enumerate}[{\normalfont (i)}]
\item $\gcd (c^{k}-1,c^{\ell }-1)=c^{\gcd (k,\ell )}-1$,
\item
$\gcd (c^{k}+1,c^{\ell }+1)=
\begin{cases}
c^{\gcd (k,\ell )}+1 & \text{ if } \nu _{2}(k)=\nu _{2}(\ell )
\\
2 & \text{ if } \nu _{2}(k)\neq \nu _{2}(\ell ) \text{ and } c
\text{ is odd}
\\
1 & \text{ if } \nu _{2}(k)\neq \nu _{2}(\ell ) \text{ and } c
\text{ is even},
\\
\end{cases}
$
\item
$\gcd (c^{k}+1,c^{\ell }-1)=
\begin{cases}
c^{\gcd (k,\ell )}+1 & \text{ if } \nu _{2}(k)<\nu _{2}(\ell )
\\
2 & \text{ if } \nu _{2}(k)\geq \nu _{2}(\ell ) \text{ and } c
\text{ is odd}
\\
1 & \text{ if } \nu _{2}(k)\geq \nu _{2}(\ell ) \text{ and } c
\text{ is even. }
\\
\end{cases}
$
\end{enumerate}
\end{lemma}

\begin{proof}
%
%
%
\begin{enumerate}[(iii)]
\item[(i)] This is a classical result, so we skip the proof.
\item[(ii)] Suppose that $k\geq \ell $. We first consider the case when
$\gcd (k,\ell )=1$. Then
\begin{align*}
\gcd (c^{k}+1,c^{\ell }+1) &= \gcd (c^{k}-c^{\ell },c^{\ell }+1) = \gcd (c^{
\ell }(c^{k-\ell }-1),c^{\ell }+1)
\\
&= \gcd (c^{k-\ell }-1,c^{\ell }+1) = \gcd (c^{k-\ell }+c^{\ell },c^{\ell }+1)
\\
&= \gcd (c^{\ell }(c^{k-2\ell }+1),c^{\ell }+1) = \gcd (c^{k-2\ell }+1,c^{
\ell }+1)
\\
&= \cdots
\\
&= \gcd (c+ 1, c+1) \text{ or } \gcd (c-1,c+1)
\\
&= c+1 \text{ or } \gcd (c-1,c+1).
\end{align*}
 We note that $\gcd (c^{k}+1,c^{\ell }+1)=c+1$ if and only if both $k$ and
$\ell $ are odd. Therefore
\begin{eqnarray*}
\gcd (c^{k}+1,c^{\ell }+1)=
\begin{cases}
c+1 & \text{ if }\nu _{2}(k)=\nu _{2}(\ell )
\\
2 & \text{ if } \nu _{2}(k)\neq \nu _{2}(\ell ) \text{ and } c
\text{ is odd}
\\
1 &\text{ if } \nu _{2}(k)\neq \nu _{2}(\ell ) \text{ and } c
\text{ is even}.
\end{cases}
\end{eqnarray*}
Now suppose that $d:= \gcd (k,\ell )>1$. Then
$\gcd (c^{k}+1,c^{\ell }+1) = \gcd ((c^{d})^{k/d}+1,(c^{d})^{\ell /d}+1)$.
Since $\gcd (k/d,\ell /d)=1$, the result follows from the previous case.
\item[(iii)] Suppose that $k\geq \ell $. Let $d:=\gcd (k,\ell )$. Then
$d=\gcd (k,k-\ell )$ and
\begin{align*}
\gcd (c^{k}+1,c^{\ell }-1)& = \gcd (c^{k}+1,c^{k}+c^{\ell })
\\
&= \gcd (c^{k}+1,c^{\ell }(c^{k-\ell }+1))
\\
&= \gcd (c^{k}+1,c^{k-\ell }+1)
\\
&=
\begin{cases}
c^{d}+1 & \text{ if } \nu _{2}(k)= \nu _{2}(k-\ell )
\\
2 & \text{ if } \nu _{2}(k)\neq \nu _{2}(k-\ell ) \text{ and } c
\text{ is odd}
\\
1 & \text{ if } \nu _{2}(k)\neq \nu _{2}(k-\ell ) \text{ and } c
\text{ is even}.
\end{cases}
\\
\end{align*}
It turns out that $\nu _{2}(k)=\nu _{2}(k-\ell )$ if and only if
\begin{equation*}
\dfrac{k(k-\ell )}{d^{2}} = \dfrac{k}{d}\left (\dfrac{k}{d}-
\dfrac{\ell }{d}\right )%
\end{equation*}
is odd, which is equivalent to $\nu _{2}(k) < \nu _{2}(\ell )$. Using a
similar argument, we obtain the desired result when $k< \ell $.\qedhere
\end{enumerate}
\end{proof}

Our first family consists of R\'{e}dei permutations of the form
$R_{p^{\ell },a}$ on $\mathbb{P}^{1}(\mathbb{F}_{p^{k}})$.

\begin{proposition}%
\label{prop_p^k}
Let $q=p^{k}$ where $p$ is a prime, and
$1\leq \ell _{1}, \ell _{2} < k$. Then
$(p^{\ell _{1}},p^{\ell _{2}})\in \mathcal{S}_{\chi }^{q}$ if and only if
%
%
%
\begin{enumerate}[{\normalfont (i)}]
\item $\gcd (\ell _{1},k) = \gcd (\ell _{2},k)$ and
\item if $\chi =-1$, then either
$\nu _{2}(\ell _{1}),\nu _{2}(\ell _{2}) > \nu _{2}(k)$ or
$\nu _{2}(\ell _{1}),\nu _{2}(\ell _{2}) \leq \nu _{2}(k)$.
\end{enumerate}
In this case, when $k$ is prime, the cycle structure is
\begin{equation*}
\begin{cases}
\left (2\times \{\bullet \}\right ) \oplus \left (\dfrac{p^{2}-1}{4}
\times \mathrm{Cyc}(4)\right ) & \text{if } \chi =-1 \text{ and } k=2
\\
\left (2\times \{\bullet \}\right ){ \oplus }\left (\dfrac{p-1}{2}
{\times} \mathrm{Cyc}(2)\right ) {\oplus} \left (\dfrac{p^{k}-p}{2k}
{\times} \mathrm{Cyc}(2k)\right ) & \text{if } \chi {=}-1 \text{ and } k,
\ell _{1}, \ell _{2} \text{ are odd}
\\
\left ((p+1)\times \{\bullet \}\right ) \oplus \left (
\dfrac{p^{k}-p}{k} \times \mathrm{Cyc}(k)\right ) & \text{otherwise.}
\end{cases}
\end{equation*}
\end{proposition}
\begin{proof}
We begin by observing that $R_{p^{\ell },a}$ always induces a permutation
since $\gcd (p^{\ell },p^{k}-\chi )=1$.

Fix $r$. It is clear that $\gcd (\ell _{1},k)=\gcd (\ell _{2},k)$ if and
only if $\gcd (r\ell _{1},k)=\gcd (r\ell _{2},k)$. We treat the cases
$\chi =-1$ and $\chi =1$ separately.
\begin{enumerate}
\item[\underline{Case 1}:] $\chi =-1$%
The condition
$\gcd (p^{r\ell _{1}}-1, p^{k}+1) = \gcd (p^{r\ell _{2}}-1, p^{k}+1)$ is
equivalent to either $\nu _{2}(r\ell _{1})$,
$\nu _{2}(r\ell _{2})>\nu _{2}(k)$ and
$\gcd (\ell _{1},k)=\gcd (\ell _{2},k)$ or
$ \nu _{2}(r\ell _{1}), \nu _{2}(r\ell _{2})\leq \nu _{2}(k)$.
\begin{enumerate}
\item[\underline{Case 1.1}:]
$\nu _{2}(\ell _{1}), \nu _{2}(\ell _{2})> \nu _{2}(k)$

In this case, we have
$\gcd (p^{r\ell _{i}}-1, p^{k}+1) = p^{\gcd (r\ell _{i},k)}+1$ for
$i=1,2$, and the result follows.
\item[\underline{Case 1.2}:] $\nu _{2}(\ell _{1})$,
$\nu _{2}(\ell _{2})\leq \nu _{2}(k)$

Since $\gcd (\ell _{1},k)=\gcd (\ell _{2},k)$, it follows that
$\nu _{2}(r\ell _{1})=\nu _{2}(r\ell _{2})$, and so either
$\nu _{2}(r\ell _{1})$, $\nu _{2}(r\ell _{2})> \nu _{2}(k)$ or
$\nu _{2}(r\ell _{1}),\nu _{2}(r\ell _{2})\leq \nu _{2}(k)$. Thus
$\gcd (p^{r\ell _{1}}-1, p^{k}+1)= \gcd (p^{r\ell _{2}}-1, p^{k}+1)$ for
every positive integer $r$.

Next, we choose $r = 2^{\nu _{2}(k)+1}$ to show that the condition
$\gcd (\ell _{1},k)=\gcd (\ell _{2},k)$ is required. Otherwise, since
$\nu _{2}(r\ell _{1}), \nu _{2}(r\ell _{2})> \nu _{2}(k)$, we obtain that
\begin{equation*}
\gcd (p^{r\ell _{1}}-1, p^{k}+1) =p^{\gcd (r\ell _{1},k)}+1 \neq p^{\gcd (r\ell _{2},k)}+1 = \gcd (p^{r\ell _{2}}-1, p^{k}+1).
\end{equation*}

\item[\underline{Case 1.3}:]
$\nu _{2}(\ell _{1}) \leq \nu _{2}(k)< \nu _{2}(\ell _{2})$

In this case,
\begin{equation*}
\gcd (p^{\ell _{1}}-1,p^{k}+1) = 2 < p^{\gcd (\ell _{2},k)}+1=\gcd (p^{
\ell _{2}}-1,p^{k}+1)%
\end{equation*}
shows that the numbers of fixed points in the first iteration are not the
same.
\end{enumerate}
\item[\underline{Case 2}:] $\chi =1$%
This case follows easily since
$\gcd (p^{r\ell _{i}}-1, p^{k}-1)=p^{\gcd (r\ell _{i},k)}$ for
$i=1,2$, concluding the first part of the proof.
\end{enumerate}

We now discuss the cycle structure of $R_{p^{\ell },a}$ depending on
$\chi $, when $k$ is prime.
\begin{enumerate}
\item[\underline{Case 1}:] $\chi =-1$%
The number of fixed points in the $r^{\mathrm{th}}$ iteration of
$R_{p^{\ell },a}$ is
\begin{equation*}
\gcd (p^{r\ell }-1,p^{k}+1) =
\begin{cases}
p^{\gcd (r\ell ,k)}+1& \text{ if } \nu _{2}(r\ell )>\nu _{2}(k)
\\
2 & \text{ otherwise}.
\end{cases}
\end{equation*}
If $k=2$, then $\ell =1$ and there are $\gcd (p-1,p^{2}+1) = 2$ fixed points
in the first iteration. The number of fixed points only increases in the
fourth iteration to $\gcd (p^{4}-1,p^{2}+1) =p^{2}+1$. In this step, all
points are fixed, so the non-fixed points are distributed over
$(p^{2}-1)/4$ cycles of length $4$. Now suppose that $k$ is an odd prime.
We consider two cases.
\begin{enumerate}
\item[\underline{Case 1.1}:] $\ell $ is odd%
We have
\begin{equation*}
\gcd (p^{r\ell }-1,p^{k}+1) =
\begin{cases}
2 & \text{ if } r \text{ is odd}
\\
p^{\gcd (r\ell ,k)}+1 & \text{ if } r \text{ is even}.
\\
\end{cases}
\end{equation*}
This implies that in the first iteration there are two fixed points and
no cycle of odd length. In addition, since $\ell <k$, we have
\begin{equation*}
\gcd (r\ell ,k) =
\begin{cases}
1 & \text{ if } k\nmid r
\\
k & \text{ if } k\mid r
\\
\end{cases}
\end{equation*}
and all points are fixed in the $(2k)^{\mathrm{th}}$ iteration. So the
non-fixed points are distributed over $(p-1)/2$ cycles of length 2 and
$(p^{k}-p)/(2k)$ cycles of length $2k$.
\item[\underline{Case 1.2}:] $\ell $ is even
We have
\begin{equation*}
\gcd (p^{r\ell }-1,p^{k}+1) = p^{\gcd (r\ell ,k)}+1 =
\begin{cases}
p+1 & \text{ if } k\nmid r
\\
p^{k}+1 & \text{ if } k\mid r.
\\
\end{cases}
\end{equation*}
Hence there are $p+1$ fixed points and $(p^{k}-p)/k$ cycles of length
$k$.
\end{enumerate}
\item[\underline{Case 2}:] $\chi =1$%
By {Proposition~\ref{fixpt}}(i), the length of each cycle is given
by $o_{d}(p^{\ell })$ where $d$ is a divisor of $q-1$. Since
$p^{\ell k} \equiv 1 \pmod{d}$, each cycle length is a divisor of $k$. In
particular, if $k$ is prime then the cycles have length 1 or $k$. The number
of fixed points is
$\gcd (p^{\ell }-1,p^{k}-1) +2 = p^{\gcd (\ell ,k)} -1+2 = p+1$. There are
$p^{k}-p$ points left, and thus $(p^{k}-p)/k$ cycles of length $k$.\qedhere
\end{enumerate}
\end{proof}

The argument used in the last Case 2 shows that each cycle length of
$R_{p^{\ell },a}$ over $\mathbb{P}^{1}(\mathbb{F}_{p^{k}})$ is a divisor of
$k$ or $2k$, depending on whether $\chi (a)=1$ or $-1$, respectively.

\begin{proposition}%
\label{isomp2kchi}
Let $q$ be a power of a prime $p$. Then
$(p,q-p+1)\in \mathcal{S}_{\chi }^{q}$ if and only if either
$\chi =-1$ and $q=p^{2k}$ for some integer $k$ or $\chi =1$ and
$q\in \{9,p\}$. In this case, the cycle structure is
\begin{equation*}
\begin{cases}
\left (2\times \{\bullet \}\right ) \oplus \displaystyle \bigoplus _{
\genfrac{}{}{0pt}{2}{d\mid 2k}{\nu _{2}(d)=\nu _{2}(2k)} } \left (N_{2d}
\times \mathrm{Cyc}(2d)\right ) &\text{ if } \chi =-1 \text{ and }
q=p^{2k}
\\
\left (4\times \{\bullet \}\right ) \oplus \left (3\times
\mathrm{Cyc}(2)\right )& \text{ if } \chi =1 \text{ and } q=9
\\
\left ((p+1)\times \{\bullet \}\right )& \text{ if } \chi =1
\text{ and } q=p,
\end{cases}
\end{equation*}
where $N_{i}$ is the number of $i$-cycles and
%
\begin{equation}
\label{cycles}
N_{2d} =\dfrac{1}{2d} \left (p^{d}+1 - \displaystyle \sum _{
\genfrac{}{}{0pt}{2}{s\mid 2d, \; s<2d}{\nu _{2}(s)=\nu _{2}(2d)} } 2sN_{2s}
-2\right ).
\end{equation}
\end{proposition}

\begin{proof}
We consider four cases.
\begin{enumerate}
\item[\underline{Case 1}:] $q=p$

We have $\gcd (0,p-\chi )+1+\chi =p+1$. When $\chi =1$, we have
$\gcd (p^{r}-1,p-1)+2= p+1$, so $(p,1)\in \mathcal{S}_{1}^{p}$. When
$\chi =-1$ and $r=1$, we have $\gcd (p-1,p+1)= 2\neq p+1$, so
$(p,1)\notin \mathcal{S}_{-1}^{p}$
\item[\underline{Case 2}:] $q=p^{k}$ with $k>1$ and $\chi =1$

We first consider the case $p=3$. When $r=2$, we have that
\begin{equation*}
\gcd (3^{2}-1,3^{k}-1) =
\begin{cases}
2 & \text{ if } k \text{ is odd}
\\
8 & \text{ if } k \text{ is even}
\\
\end{cases}
\end{equation*}
and $\gcd ((3^{k}-3+1)^{2}-1,3^{k}-1) = 3^{k}-1$ are different if and only
if $k\neq 2$. The result for $k=2$ follows from {Proposition~\ref{fixpt}}(i).
When $p\neq 3$, take $r=k$. We have that
$\gcd (p^{k}-1,p^{k}-1) = p^{k}-1$ and
\begin{equation*}
\gcd ((p^{k}-p+1)^{k}-1,p^{k}-1) =
\begin{cases}
\gcd ((p-2)^{k}+1,p^{k}-1) & \text{ if } k \text{ is odd}
\\
\gcd ((p-2)^{k}-1,p^{k}-1) & \text{ if } k \text{ is even}
\\
\end{cases}
\end{equation*}
are different, since it can be shown by induction that
$p^{k}-1> (p-2)^{k}+1$.
\item[\underline{Case 3}:] $q=p^{2k+1}$ and $\chi =-1$

For $r=1$, we have that
\begin{equation*}
\gcd (p-1,p^{2k+1}+1) = 2%
\end{equation*}
and
\begin{equation*}
\gcd ((p^{2k+1}-p+1)-1,p^{2k+1}+1) =\gcd (p+1,p^{2k+1}+1) =p+1
\end{equation*}
are different.

\item[\underline{Case 4}:] $q=p^{2k}$ and $\chi =-1$

We have that
\begin{align*}
\gcd ((q-p+1)^{r}-1,q+1) &= \gcd ((-p)^{r}-1,p^{2k}+1)
\\
&=
\begin{cases}
\gcd (p^{r}+1,p^{2k}+1)=2 & \text{ if } r \text{ is odd}
\\
\gcd (p^{r}-1,p^{2k}+1) & \text{ if } r \text{ is even}.
\end{cases}
\end{align*}
When $r$ is odd, we also have that $\gcd (p^{r}-1,p^{2k}+1)=2$. Therefore
$(p,p^{2k}-p+1)\in S_{-1}^{p^{2k}}$. In this case, the number of fixed
points in the $r^{\mathrm{th}}$ iteration is
\begin{align*}
\gcd (p^{r}-1,p^{2k}+1) =
\begin{cases}
2 & \text{ if } \nu _{2}(r) \leq \nu _{2}(2k)
\\
p^{\gcd (r,2k)}+1 & \text{ if } \nu _{2}(r) >\nu _{2}(2k).
\end{cases}
\end{align*}
In particular, there is no cycle of odd length other than one. If
$r$ is even, then there are more than two fixed points in the
$r^{\mathrm{th}}$ iteration if and only if
$\nu _{2}(r)> \nu _{2}(2k)$. In this case, the number of fixed points is
$p^{d}+1$ where $d = \gcd (r,2k)$ with $\nu _{2}(d) = \nu _{2}(2k)$. The
first iteration with $p^{d}+1$ fixed points corresponds to $r = 2d$. Thus,
the graph has $2d$-cycles for each divisor $d$ of $2k$ such that
$\nu _{2}(d) = \nu _{2}(2k)$. Then
\begin{equation*}
2d N_{2d} = p^{d}+1 - \sum _{
\genfrac{}{}{0pt}{2}{s\mid 2d,\; s<2d}{\nu _{2}(s)=\nu _{2}(2d)} } 2sN_{2s}
-2.\qedhere
\end{equation*}
\end{enumerate}
\end{proof}

The following example illustrates the previous result.

\begin{example}
\label{exmp5.4}
By {Proposition~\ref{isomp2kchi}}, both $R_{3,a}$ and $R_{3^{60}-2,b}$ have
the same cycle structure on
$\mathbb{P}^{1}\left (\mathbb{F}_{3^{60}}\right )$ when
$\chi (a)=\chi (b)=-1$. To obtain the number of cycles and their corresponding
lengths, we apply~{\eqref{cycles}} to every positive divisor $d$ of
$2k=60$ with $\nu _{2}(d)=\nu _{2}(60)=2$.
\begin{itemize}
\item[-] For $d=4$, we get $N_{8} = \left (3^{4} +1-2\right )/8$, so
$N_{8}=10$.
\item[-] For $d=12$, we get
$N_{24} =\left (3^{12} +1-8N_{8}-2\right )/24$, so $N_{24}=22,140$
\item[-] For $d=20$, we get
$N_{40} = \left (3^{20} +1-8N_{8}-2\right )/40$, so
$N_{40}=87,169,608$.
\item[-] For $d=60$, we get
$N_{120} =\left (3^{60} +1-8N_{8}-24N_{24} -40N_{40}-2\right )/120$, so
\begin{equation*}
N_{120}=353,259,652,293,468,362,590,059,312.%
\end{equation*}
\end{itemize}
Hence the cycle structure is
\begin{align*}
& \left (2\times \{\bullet \}\right ) \oplus \left (10 \times
\mathrm{Cyc}(8)\right ) \oplus \left (22,140 \times \mathrm{Cyc}(24)
\right ) \oplus \left (87,169,608 \times \mathrm{Cyc}(40)\right )
\\
\oplus & \left (353,259,652,293,468,362,590,059,312 \times
\mathrm{Cyc}(120)\right ).
\end{align*}
\end{example}

We observe that $q=p^{k}$ satisfying $q\equiv 1 \pmod{8}$ is equivalent
to having either $p\equiv 1\pmod{8}$ or $k$ even. In particular, it applies
to any $q$ that is an even power of an odd prime.

\begin{proposition}%
\label{propp^2k}
If $q \equiv \chi \pmod{8}$, then
\begin{equation*}
\left (\frac{q-\chi }{4}+1,\frac{3(q-\chi )}{4}+1\right )\in S_{\chi }^{q}.%
\end{equation*}
In this case, the cycle structure is
\begin{eqnarray*}
\begin{cases}
\left (\left (\dfrac{q-\chi }{4} +\chi +1\right )\times \{\bullet \}
\right ) \oplus \left (\dfrac{3(q-\chi )}{8} \times \mathrm{Cyc}(2)
\right ) & \text{if } \dfrac{q-\chi }{8} \text{ is odd}
\\
\left (\left (\dfrac{q-\chi }{4} {+}\chi {+}1\right ){\times} \{\bullet \}
\right ){\oplus} \left (\dfrac{q-\chi }{8} {\times} \mathrm{Cyc}(2)\right )
\oplus \left (\dfrac{q-\chi }{8} {\times} \mathrm{Cyc}(4)\right ) &
\text{if } \dfrac{q-\chi }{8} \text{ is even}.
\end{cases}
\end{eqnarray*}
\end{proposition}

\begin{proof}
If $q \equiv \chi \pmod{8}$, then $\gcd ((q-\chi )/4+1,q-\chi )=1$ and
$\gcd ((3(q-\chi ))/4+1,q-\chi )=1$, so both
$R_{\frac{q-\chi }{4}+1,a}$ and $R_{\frac{3(q-\chi )}{4}+1,b}$ induce permutations
when $\chi (a)=\chi (b)=\chi $. Their cycle structures are the same if
and only if
%
\begin{equation}
\label{twogcds}
\gcd \left ( \left (\displaystyle \frac{q-\chi }{4}+1\right )^{r}-1,q-
\chi \right ) = \gcd \left (\left (\displaystyle \frac{3(q-\chi )}{4}+1
\right )^{r}-1,q-\chi \right )
\end{equation}
for every positive integer $r$. We have
\begin{eqnarray*}
\left (\displaystyle \frac{q-\chi }{4}+1\right )^{r}-1 =
\displaystyle \sum _{i=1}^{r} \binom{r}{i} \left (\displaystyle
\frac{q-\chi }{4} \right )^{i}.
\end{eqnarray*}
By writing $(q-\chi )/4=2k$ for some integer $k$ and
\begin{equation*}
\left (\displaystyle \frac{q-\chi }{4} \right )^{i} = \frac{q-\chi }{4}
\left (\displaystyle \frac{q-\chi }{4} \right )^{i-1} = (q-\chi )
\cdot \frac{(2k)^{i-1}}{4},%
\end{equation*}
it follows that $\left (\displaystyle \frac{q-\chi }{4} \right )^{i}$ is
a multiple of $q-\chi $ for $i\geq 3$. We conclude that
%
\begin{align}
\gcd \left ( \left (\displaystyle \frac{q-\chi }{4}+1\right )^{r}-1,q-
\chi \right ) & = \gcd \left ( r\cdot \frac{q-\chi }{4} + r(r-1)\cdot
\frac{(q-\chi )^{2}}{32} ,q-\chi \right )
\nonumber
\\
& = \frac{q-\chi }{4}\cdot \gcd \left ( r\left (1+(r-1)\cdot
\frac{q-\chi }{8}\right ),4\right ).
\label{gcd1}
\end{align}
Analogously, we have
\begin{eqnarray*}
\left (\displaystyle \frac{3(q-\chi )}{4}+1\right )^{r}-1 =
\displaystyle \sum _{i=1}^{r} \binom{r}{i} 3^{i}\left (\displaystyle
\frac{q-\chi }{4} \right )^{i},
\end{eqnarray*}
so
%
\begin{align}
\gcd \left (\left (\displaystyle \frac{3(q-\chi )}{4}+1\right )^{r}-1
,q-\chi \right ) = & \gcd \left (3r\cdot \frac{q-\chi }{4} + 9r(r-1) \cdot
\frac{(q-\chi )^{2}}{32} ,q-\chi \right )
\nonumber
\\
= & \frac{q-\chi }{4}\cdot \gcd \left (3r\left (1+3(r-1)\cdot
\frac{q-\chi }{8}\right ),4\right )
\nonumber
\\
= & \frac{q-\chi }{4}\cdot \gcd \left (r\left (1+3(r-1)\cdot
\frac{q-\chi }{8}\right ),4\right )
\label{gcd2}
.
\end{align}
The greatest common divisors in~{\eqref{gcd1}} and~{\eqref{gcd2}} are equal
to
\begin{equation*}
\label{gcdvalue}
\begin{cases}
1 & \text{if } \nu _{2}(r)=0
\\
2 & \text{if } \nu _{2}(r)=1 \text{ and } \dfrac{q-\chi }{8}
\text{ is even}
\\
4 & \text{if either } \nu _{2}(r)\geq 2\text{ or } \nu _{2}(r)=1
\text{ and } \dfrac{q-\chi }{8} \text{ is odd}.
\end{cases}
\end{equation*}
 In any event, we have that {\eqref{twogcds}} holds for any positive integer
$r$. When $r$ is odd, the number of fixed points in any
$r^{\mathrm{th}}$ iteration is $(q-\chi )/4 +\chi +1$, and so the only
cycle of odd length has length one. In the second iteration, the number
of fixed points is
\begin{equation*}
\begin{cases}
q+1& \text{if } \dfrac{q-\chi }{8} \text{ is odd}
\\
\dfrac{q-\chi }{2} +\chi +1& \text{if } \dfrac{q-\chi }{8}
\text{ is even.}
\end{cases}
\end{equation*}
If $(q-\chi )/8$ is odd, the R\'{e}dei function is an involution with
$(3(q-\chi ))/8$ cycles of length 2. Otherwise, there are
$(q-\chi )/8$ cycles of length 2. In this case, all points are fixed in
the fourth iteration, and so the remaining points are distributed over
$(q-\chi )/8$ cycles of length 4.
\end{proof}

\begin{proposition}%
\label{qcongchipm2}
If $q \equiv \chi \pm 2 \pmod{8}$, then
\begin{equation*}
\left (\dfrac{q-\chi \pm 2}{4},\dfrac{q-\chi \pm 4}{2}\right )\in S_{
\chi }^{q}.%
\end{equation*}
\end{proposition}

\begin{proof}
If $q \equiv \chi \pm 2 \pmod{8}$, then $(q-\chi \pm 2)/4$ and
$(q-\chi \pm 4)/2$ are odd and $q-\chi $ is even. If
$d \mid (q-\chi \pm 4)/2$ and $d \mid q-\chi $, then
$d \mid 2\left ( \frac{q-\chi \pm 4}{2}\right ) -(q-\chi ) = \pm 4$. Also,
if $d' \mid (q-\chi \pm 2)/4$ and $d' \mid q-\chi $, then
$d' \mid 4\left (\frac{q-\chi \pm 2}{4}\right ) - (q-\chi ) = \pm 2$. Thus
$\gcd ((q-\chi \pm 4)/2,q-\chi )=1$ and
$\gcd ((q-\chi \pm 2)/4,q-\chi )=1$ imply that both
$R_{(q-\chi \pm 2)/4,a}$ and $R_{(q-\chi \pm 4)/2,b}$ induce permutations.
Their cycle structures are the same if and only if
\begin{equation}
\label{twogcdsv1}
\gcd \left ( \left (\displaystyle \frac{q-\chi \pm 4}{2}\right )^{r}-1,q-
\chi \right ) = \gcd \left (\left (\displaystyle
\frac{q-\chi \pm 2}{4}\right )^{r}-1,q-\chi \right )
\end{equation}
for every positive integer $r$. Write $q-\chi =8k\pm 2$ for some integer
$k$. Then $(q-\chi \pm 2)/4=2k\pm 1$ and $(q-\chi \pm 4)/2=4k\pm 3$. Since
$q-\chi = 2(4k\pm 1)$, it follows that
\begin{align*}
\gcd \left ( \left (\dfrac{q-\chi \pm 2}{4}\right )^{r}-1,q-\chi
\right )
=&\gcd \left ( \left (2k\pm 1\right )^{r}-1,2\right )\cdot \gcd
\left ( \left (2k\pm 1\right )^{r}-1,4k\pm 1\right )
\\
=&2 \gcd \left ( \left (2k\pm 1\right )^{r}-1,4k\pm 1\right )
\end{align*}
and
\begin{align*}
\gcd \left ( \left (\dfrac{q-\chi \pm 4}{2}\right )^{r}-1,q-\chi
\right )
=&\gcd \left ( \left (4k\pm 3\right )^{r}-1,2\right )\cdot \gcd
\left ( \left (4k\pm 3\right )^{r}-1,4k\pm 1\right )
\\
= &2 \gcd \left ( \left (4k\pm 3\right )^{r}-1,4k\pm 1\right ).
\end{align*}
Furthermore, since $2k\pm 1 \equiv -2k \pmod{4k\pm 1}$ and
$4k\pm 3 \equiv \pm 2 \pmod{4k\pm 1}$, we have
\begin{equation*}
\gcd \left ( \left (2k\pm 1\right )^{r}-1,4k\pm 1\right ) = \gcd
\left ( \left (-2k\right )^{r}-1,4k\pm 1\right )%
\end{equation*}
and
\begin{equation*}
\gcd \left ( \left (4k\pm 3\right )^{r}-1,4k\pm 1\right ) = \gcd
\left ( (\pm 2)^{r}-1,4k\pm 1\right ).%
\end{equation*}
Hence {\eqref{twogcdsv1}} is equivalent to
\begin{equation*}
\label{twogcdsv2}
\gcd ((-2k)^{r}-1,4k\pm 1)=\gcd ((\pm 2)^{r}-1,4k\pm 1).
\end{equation*}
This equality holds for $r=1$ as both gcds are either $1$ or $3$. For
$r\geq 2$, we claim that if $d\mid 4k\pm 1$, then
$d\mid (-2k)^{r}-1$ if and only if $d\mid (\pm 2)^{r}-1$. First, suppose
$d\mid 4k\pm 1$ and $d\mid (-2k)^{r}-1$. Then
$4k\equiv \mp 1\pmod{d}$, so $(4k)^{r}\equiv (\mp 1)^{r}\pmod{d}$. On
the other hand,
\begin{equation*}
(4k)^{r}=(-2)^{r}(-2k)^{r}\equiv (-2)^{r}\pmod{d}.%
\end{equation*}
Therefore $(-2)^{r}\equiv (\mp 1)^{r}\pmod{d}$, so
$(\pm 2)^{r}\equiv 1 \pmod{d}$. Conversely, suppose $d\mid 4k\pm 1$ and
$d\mid (\pm 2)^{r}-1$. Then
$d\mid 4k\pm (\pm 2)^{r} = 4(k\pm (\pm 2)^{r-2})$ implies that
$d\mid k\pm (\pm 2)^{r-2}$ since $d$ is odd. So
$-k\equiv \pm (\pm 2)^{r-2}\pmod{d}$ and
$(\pm 2)^{r}\equiv 1 \pmod{d}$ give that
\begin{equation*}
(-2k)^{r}=(-k)^{r}2^{r}\equiv (\pm (\pm 2)^{r-2})^{r}2^{r}\equiv 1
\pmod{d}.%
\qedhere \end{equation*}
\end{proof}

\section{Concluding remarks}
\label{conclusao}

We recall that $\mathcal{S}_{\chi }^{q}$ is the set of all
$(m,n)\in \mathbb{N}^{2}$ such that $R_{m,a}$ and $R_{n,b}$ are R\'{e}dei
permutations with the same cycle structure for some
$a,b\in {\mathbb{F}}_{q}$ with $\chi (a)=\chi (b)=\chi $. In this paper
we provide a criterion for a pair to belong to $S_{\chi }^{q}$ that can
be used to generate all R\'{e}dei permutations $R_{m,a}$ with the same cycle
structure, subject to $a$ being a square or not in $\mathbb{F}_{q}$. We
also completely describe the counterpart, consisting of R\'{e}dei permutations
that are isolated, by showing that they are precisely the isolated R\'{e}dei
involutions. This provides a link with our previous work, where we establish
a formula for $m$ such that $R_{m,a}$ is an isolated involution, and for
all pairs of non-isolated involutions in $S_{\chi }^{q}$. If $(m,n)$ and
$(m,n')$ are pairs of involutions in $S_{\chi }^{q}$, then
$n'\equiv m \text{ or } n \pmod{q-\chi }$. In other words, there are at most
two involutions with the same cycle structure.

When $a$ is a square in $\mathbb{F}_{q}$, we let
$\mathbb{D}_{q}=\mathbb{P}^{1}({\mathbb{F}}_{q}) \setminus \{\pm \sqrt{a}
\}$. Since $-\sqrt{a}$ and $\sqrt{a}$ are fixed points of $R_{m,a}$, the
cycle structures of $R_{m,a}$ over $\mathbb{P}^{1}({\mathbb{F}}_{q})$ and
$\mathbb{D}_{q}$ are the same. Let
$f_{m}\colon \mathbb{Z}_{k} \to \mathbb{Z}_{k}$ be the mapping defined by
$f_{m}(x)=mx$, and $U_{q+1}$ be the multiplicative subgroup of order
$q+1$ in $\mathbb{F}_{q^{2}}$. In~\cite{MR3384830} it is shown that the
cycle structures of $R_{m,a}$ over $\mathbb{P}^{1}({\mathbb{F}}_{q})$,
$x^{m}$ over $U_{q+1}$, and $f_{m}$ over $\mathbb{Z}_{q+1}$ are the same
when $\chi (a)=-1$. In addition, the cycle structures of $R_{m,a}$ over
$\mathbb{D}_{q}$, $x^{m}$ over $\mathbb{F}_{q}^{*}$, and $f_{m}$ over
$\mathbb{Z}_{q-1}$ are the same when $\chi (a)=1$. This allows us to transfer
all our results in this paper to these settings. To make this more precise,
we define the following three sets:
\begin{align*}
\mathcal{T}_{\chi }^{q} &{=}\{ (m,n)\in \mathbb{N}^{2} \colon f_{m}
\text{ and } f_{n} \text{ are bijections over } \mathbb{Z}_{q-\chi }
\text{ with the same cycle structure}\},
\\
\mathcal{U}_{-1}^{q} &{=} \{(m,n)\in \mathbb{N}^{2}\colon x^{m}
\text{ and } x^{n} \text{ are bijections over } U_{q+1}
\text{ with the same cycle structure}\},
\\
\mathcal{U}_{1}^{q} &{=}\{(m,n)\in \mathbb{N}^{2}\colon x^{m}
\text{ and } x^{n} \text{ are bijections over } \mathbb{F}_{q}^{*}
\text{ with the same cycle structure}\}.
\end{align*}
It follows that
$\mathcal{S}_{\chi }^{q} = \mathcal{T}_{\chi }^{q} = \mathcal{U}_{\chi }^{q}$.
Consequently, all results that hold for $\mathcal{S}_{\chi }^{q}$ also hold
for $\mathcal{T}_{\chi }^{q}$ and $\mathcal{U}_{\chi }^{q}$. Moreover, by {Corollary~\ref{coroisolated}}
the only functions $f_{m}$ and $x^{m}$ that are isolated in their respective
domains are the isolated involutions.



\section{Acknowledgements}
The authors would like to thank Daniel Panario, Claudio Qureshi, and Satyanand
Singh for helpful discussions. The second author received support for this
project provided by a
PSC-CUNY grant, jointly funded by
The Professional Staff Congress and The City University of New York.

%
%

------------------------------------------------------------------------------------------

\end{document}